\documentclass[english,11pt]{article}

\usepackage[english]{babel}
\usepackage{amsmath,amsthm,amssymb,enumerate}
\usepackage[latin1]{inputenc}
\usepackage[T1]{fontenc}
\usepackage{psfig,labelfig}
\usepackage{graphicx}
\usepackage[dvips]{epsfig}
\usepackage[color=white]{todonotes}
\usepackage{hyperref}
\usepackage{lineno}

\usepackage[center, labelfont={normalfont,bf},singlelinecheck=false]{caption}

\newtheorem{thm}{Theorem}[section]

\newtheorem{lem}[thm]{Lemma}

\newcommand{\Z}{{\mathbb Z}}
\newcommand{\N}{{\mathbb N}}

\newcommand{\hex}{\mathop{Hex}}

\DeclareMathOperator{\sy}{sys}

\DeclareMathOperator{\sr}{SR}

\DeclareMathOperator{\area}{area}
\DeclareMathOperator{\girth}{girth}
\DeclareMathOperator{\genus}{genus}

\DeclareMathOperator{\cat}{CAT}

\newcommand{\cG}{{\mathcal G}}
\newcommand{\cT}{{\mathcal T}}
\newcommand{\cH}{{\mathcal H}}

\title{Translation surfaces with large systoles}

\author{Peter Buser, Eran Makover and Bjoern Muetzel}

\begin{document}
\maketitle
\begin{abstract}
In this paper we continue to investigate the systolic landscape of translation surfaces started in \cite{chms}. We show that there is an infinite sequence of surfaces $(S_{g_k})_k$ of genus $g_k$, where $g_k \to \infty$ with large systoles. On the other hand we show that for hyperelliptic surfaces we can find a suitable homology basis, where a large number of loops that induce the basis are short.    \\
\\
Keywords: translation surfaces, systoles, maximal surfaces, hyperelliptic involution.\\
Mathematics Subject Classification (2020): 30F10, 32G15 and 53C22.
\end{abstract}

\section{Introduction}
A \textit{translation surface} $S$ is a connected surface obtained by identifying pairwise by translations the sides of a set of plane polygons. It follows that $S$ is a closed flat surface with singularities or cone points, where a cone point $p$ has angle $2\pi \cdot (k+1)$ for $k 
\in \N$ (\cite{ma},\cite{wr} and \cite{zo}). 

Translation surfaces are $\cat(0)$ spaces, where the cone points can be seen as points of "concentrated" negative curvature. Therefore these surfaces share many properties of surfaces with non-positive curvature and especially the well-studied hyperbolic surfaces (see, for example, \cite{ak}, \cite{am1},  \cite{ba}, \cite{ge}, \cite{pa1}, \cite{ra}, \cite{sc1} and \cite{sc2}). In this article we are interested in short curves on general and hyperelliptic  translation surfaces. The leitmotif is to carry over the corresponding results for compact hyperbolic surfaces to the category of translation surfaces.

More precisely, we are interested in the \textit{systole} of a translation surface $S$ which is a shortest simple closed geodesic. We denote by $\sy(S)$ its length. To make this a scaling invariant we have to normalize by the area $\area(S)$ of $S$ to obtain $\sr(S) = \frac{\sy(S)^2}{\area(S)}$, the \textit{systolic ratio} of $S$. Let 
\[
     \sr^{tr}(g) = \sup \{ \frac{\sy(S)^2}{\area(S)} \mid S \text{ translation surface of genus } g \}
\]
be the \textit{supremal systolic ratio in genus} $g$. The systolic ratio in genus $g$ can only be of order $ \frac{\log^2(g)}{g}$ (see \cite{am2}, \cite{ks}).  This estimate is valid for any smooth Riemannian surface. It also applies in the case of translation surfaces, as any translation surface can be approximated by smooth Riemannian surfaces. 

Concerning this invariant little is known in the case of translation surfaces. Preceeding this article the best lower bound known to the authors was of order $\frac{1}{g}$ (see \cite{chms}). Furthermore for fixed genus only the genus one case is clear. In the case of flat tori $\sr^{tr}(1) = \frac{2}{\sqrt{3}}$. In this case the maximal surface is the equilateral torus which has a hexagonal lattice. In genus two Judge and Parlier conjecture in \cite{jp} that the surface $\hex_2$ obtained by gluing parallel sides of two isometric regular hexagons is maximal. It is furthermore known that for fixed genus the supremal systolic ratio is a maximum (see \cite{chms}). We call a surface $S_{max}$ \textit{maximal}, if $\sr^{tr}(g)$ is attained in this surface. Maximal surfaces have the following interesting property:

\begin{thm}[\cite{chms}] Let $S_{max}$ be a maximal translation surface of genus $g \geq 1$. Then every simple closed geodesic that does not run through a cone point is intersected by a systole of $S_{max}$. 
\label{thm:char_Smax}
\end{thm} 
Concerning the lower bound of $\sr^{tr}(g)$ we show in this article in Section 2:
\begin{thm}[Surfaces with large systoles] There exist, for infinitely many $g \in \N$ translation surfaces $S_g$ of genus $g$, such that  
\[
            \frac{  \log(g)}{ g \cdot \log(\log(g)) } \leq \sr(S_g)  \leq   \sr^{tr}(g). 
\]
\label{thm:sr_low_intro}
\end{thm}

We believe that the lower bound is of order $\frac{\log(g)^2}{g}$ as conjectured for $\sr^{tr}(g)$ in \cite{chms}. Inspired by the constructions of \cite{bu} and \cite{pw} for hyperbolic surfaces, the idea for the lower bound in Theorem \ref{thm:sr_low_intro} is to construct explicit example surfaces from graphs with large girth. 

On the other hand, in the case of hyperelliptic hyperbolic surfaces it can be shown that these surfaces have many short loops. Such a characterization can also be made in the case of translation surfaces: Given a hyperelliptic translation surface $S$ of genus $g \geq 2$, we find homologically independent loops on $S$, whose lengths are bounded from above by a constant. 

\begin{thm}[Hyperelliptic surfaces] Let $S$ be a hyperelliptic translation surface of genus $g \geq 2$ and area $4\pi g$. Then there exist $\lceil  \frac{2g+2}{3}  \rceil$ geodesic loops  $(\alpha_k)_{k=1,\ldots,\left \lceil \frac{2g+2}{3}\right \rceil}$ that can be extended to a homology basis of $H_1(S,\Z)$ such that
\begin{equation*}
\ell(\alpha_{k}) \leq 4 \cdot \sqrt{\frac{ 12g}{ 2g+5-3k}} \text{ \ and \ } \frac{\ell(\alpha_{k})^2}{\area(S)} \leq  \frac{ 48}{ \pi(2g+5-3k)} \text{ for all } k=1, \ldots, \left\lceil \frac{2g+2}{3}\right \rceil.
\end{equation*}
\label{thm:hyper_intro}
\end{thm}
This means that for a hyperelliptic surface $S$ of genus $g$ and area $4\pi g$ we can find almost $\frac{2g}{3}$ loops that induce a homology basis which are shorter than a constant. These results will be presented in Section 3 and extend the result from \cite{bmm} for hyperbolic surfaces to the category of translation surfaces. These bounds hold more generally for compact surfaces endowed with a Riemannian metric of non-positive curvature that have finitely many cone-like singularities all of whose cone angles are greater than $2\pi$. This is stated in Theorem \ref{thm:hyper_section}.

\section{Translation surfaces with large systoles} \label{sec:systoles}

To construct translation surfaces with large systoles we base our construction on graphs with large girth which are well-known and studied. Here we will use $k$-regular graphs. A graph is called $k$-regular if all of its vertices have degree $k$. The idea is to identify the vertices with $k$-holed flat tori. The edges of the graph then provide the pasting scheme on how to glue the holes of the tori together. We then adjust the parameters, such that the shortest non-trivial loop in one of the tori is as large as the shortest loop going through the graph. This way we obtain:

\begin{thm}[Surfaces with large systoles] There exist, for infinitely many $g \in \N$ translation surfaces $S_g$ of genus $g$, such that  
\[
            \frac{  \log(g)}{ g \cdot \log(\log(g)) }  \leq   \sr(S_g). 
\]
\label{thm:sr_low}
\end{thm}
This is Theorem \ref{thm:sr_low_intro} from the introduction. 

\subsection{Construction of surfaces with large systoles} \label{subsec:construct}

We will start with the construction and the properties of the $k$-holed tori. To this end we start by constructing a rectangle $R$ of length $L + 2/L$ and width $2/L$ with a slit of length $L \in \N$ (see Figure \ref{fig:rect_const}).  

\begin{figure}[!ht]
\SetLabels
\L(.23*.69) $1/L$\\
\L(.23*.25) $1/L$\\
\L(.30*.91) $1/L$\\
\L(.30*.01) $1/L$\\
\L(.49*.91) $L$\\
\L(.49*.01) $L$\\
\L(.66*.91) $1/L$\\
\L(.66*.01) $1/L$\\
\L(.73*.69) $1/L$\\
\L(.73*.25) $1/L$\\
\endSetLabels
\AffixLabels{%
\centerline{%
\includegraphics[height=4cm,width=8cm]{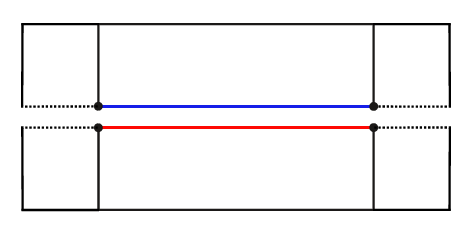}}}
\caption{Construction of a rectangle $R$ with a slit of length $L$. Dashed lines are glued together.}
\label{fig:rect_const}
\end{figure}
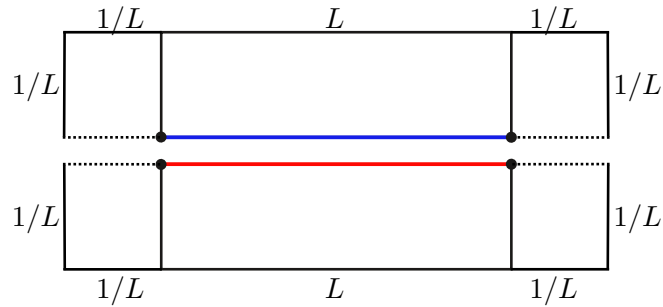 
To construct $R$ we proceed as indicated in the figure. We take two rectangles of length $L + 2/L$ and width $1/L$ and glue them together by translational identification, leaving a slit of length $L$ in the middle to obtain our slit rectangle $R$. \\
\\
\textbf{Example 1} We illustrate the later procedure with the following simple example of  a translation surface $X$ of genus two with two cone points of angle $4\pi$. This is a special case of the example given in \cite[Figure 6]{wr}. To this end we first construct a slit torus $P$ from $R$ by pasting together opposite sides of the rectangle $R$ as can be seen in Figure  \ref{fig:surface_X}. We then paste two copies $P_1$ and $P_2$ of $P$ from Figure \ref{fig:surface_X} in the following way. The upper segment of the slit in $P_1$ is pasted to the lower segment of the slit in $P_2$ and vice versa. 
\begin{figure}[!ht]
\SetLabels
\L(.32*.88) $P_1$\\
\L(.68*.88) $P_2$\\
\endSetLabels
\AffixLabels{%
\centerline{%
\includegraphics[height=2.5cm,width=11cm]{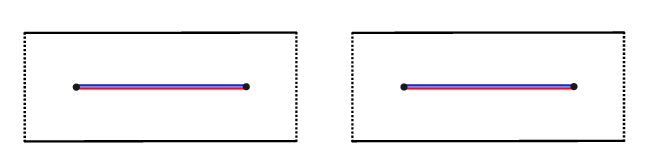}}}
\caption{Dashed and full lines are identified in each of the two rectangles to obtain the two tori $P_1$ and $P_2$. These are glued together along the slit to obtain the surface $X$.}
\label{fig:surface_X}
\end{figure}

\bigskip

We now build a flat torus $T$ with $2L^2$ slits from $2L^2$ copies of $R$. Such a torus will represent a vertex $v$ of our $k$-regular graph $\Gamma$, where 
\begin{equation}
k = 2L^2. 
\label{eq:k_Lsquared}
\end{equation}
\begin{figure}[!ht]
\SetLabels
\L(.29*.45) $\,2L$\\
\L(.45*.95) $2L+4/L$\\
\L(.70*.45) $T$\\
\endSetLabels
\AffixLabels{%
\centerline{%
\includegraphics[height=6cm,width=7.5cm]{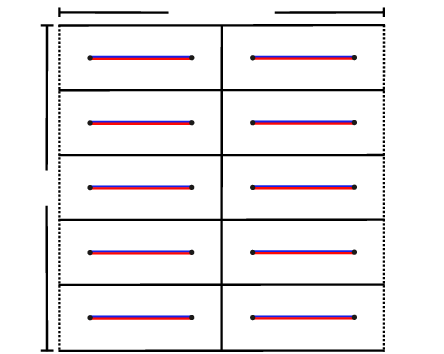}}}
\caption{Construction of a torus $T$ with $2L^2$ slits from copies of the rectangle $R$.}
\label{fig:torus_const}
\end{figure} 

This  construction is illustrated in Figure \ref{fig:torus_const}. To this end we first glue $L^2$ copies of $R$ together with a top to bottom pasting to obtain a column of rectangles $C$ with $L^2$ slits. Then we take two such copies of $C$ and paste them with a left to right identification to obtain the rectangle $T'$ with $2L^2$ slits. Finally we glue the top and the bottom of $T'$ and the left and right side of $T'$ to obtain the torus $T$ with $2L^2$ slits. $T'$ has 
\[
     \text{length \ } 2L + 4/L \text{ \  and width \ } 2L. 
\]
We now contruct a surface $S$ with the help of a $k = 2L^2$ regular graph $\Gamma$. A copy $T_i$ of $T$ represents a vertex $v_i$ of our graph $\Gamma$ and a slit of $T_i$ a half-edge of the vertex $v_i$. The pasting pattern for the slits is provided by the edges of $\Gamma$. 
More precisely, for each vertex $v_i$ of $\Gamma$ and a copy $T_i$ of $T$ we paste slit number $s$ of $T_i$ to slit number $t$ of $T_j$ if and only if half-edge $s$ of vertex $v_i$ is glued to half-edge $t$ of vertex $v_j$ thus forming an edge $e= \{s,t\}$ of $\Gamma$. We call $\Gamma$ the \textit{underlying pasting scheme} or \textit{pasting pattern} of $S$. The pasting of the slits is as in Example 1. The upper segment of the first slit is pasted to the lower segment of the second slit and vice versa. This way all cone points are of angle $4\pi$. \\
In order to get $S$ with a large systole we need a graph $\Gamma$ with a large \textit{girth}, that is, a large length of the shortest non-trivial closed reduced edge path. An edge path is called \textit{reduced} if it has no edge followed by its inverse (and in the closed case, the last edge is not the inverse of the first). The corresponding graphs with large girth will be provided in the following subsection to complete the length estimates. We now identify $\sy(S)$.

\begin{lem} If our graph $\Gamma$ in the construction of $S$ satisfies
\begin{equation}
          \girth(\Gamma) \cdot \frac{2}{L} \geq 2L, \text{ \ then \ } \sy(S) = 2L.  
\label{eq:girth_systorus}          
\end{equation}
\label{lem:sys2L}
\end{lem}

This implies that there are systoles of $S$ on every copy $T_i$ of the torus $T$.

\begin{proof}[Proof of Lemma \ref{lem:sys2L}]

Assume that our graph $\Gamma$ in the construction satisfies the condition of the lemma. 

\begin{figure}[!ht]
\SetLabels
\L(.32*.88) $T_i$\\
\L(.23*.32) $p$\\
\L(.33*.48) $c$\\
\L(.41*.32) $q$\\
\L(.68*.88) $T_i$\\
\L(.58*.32) $p$\\
\L(.67*.48) $\,\,c$\\
\L(.75*.32) $\,q$\\
\endSetLabels
\AffixLabels{%
\centerline{%
\includegraphics[height=2.5cm,width=11cm]{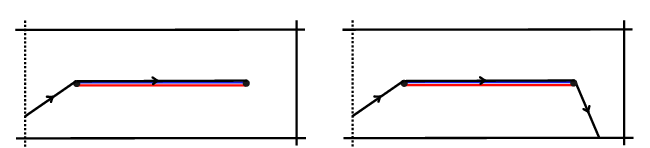}}}
\caption{An arc $c$ of the geodesic $\gamma$ connecting two cone points $p$ and $q$.}
\label{fig:geod_slit}
\end{figure}
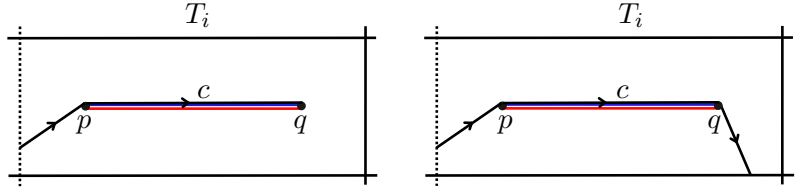 

To any simple closed geodesic $\gamma$ on $S$ we associate an edge path $\xi(\gamma)$ in the standard way: If $\gamma$ begins on copy $T_{i_0}$ of $T$, then $\xi$ begins on vertex $v_{i_0}$. Then, successively, whenever $\gamma$ leaves some $T_i$ at slit $s$ entering the adjacent $T_j$ at slit $t$ we add edge $e = \{s,t\}$ to the edge path. For the case when some subarc of $\gamma$ runs along a slit say of $T_i$ and thus is simultaneously visible on $T_i$ and on the adjacent $T_j$ we must make precise at which point $\gamma$ shall be considered entering $T_j$. Since $\gamma$ is a geodesic it runs along an entire side of the slit from one singular point to the other. Figure \ref{fig:geod_slit} shows two cases: on the left hand side subarc $c$ enters a rectangle of $T_i$ arrives at the slit on the cone point $p$ goes along the slit and leaves $T_i$ into $T_j$ at the other singular point  $q$. In this case $q$ will be the exit point. On the right hand side $c$ also goes from $p$ to $q$ but then continues in $T_i$. In this case $c$ is not considered being entering $T_j$. A particular case occurs when $\gamma$ runs entirely along the slit from $p$ to $q$ and then back to $p$. In this case $\xi(\gamma)$ is the point path $\{v_i\}$. For $\xi(\gamma)$ there are therefore two cases:

\begin{itemize}
\item[Case 1.)] $\xi(\gamma)$ is  homotopic in $\Gamma$ to a non-trivial cycle: In this case $\xi(\gamma)$ has combinatorial length greater or equal to $\girth(\Gamma)$ and $\gamma$ contains at least $\girth(\Gamma)$ arcs with disjoint interiors each of which connects two distinct slits of some $T_i$ . By the following Lemma \ref{lem:slit_connect} we have
\begin{equation}
     \ell(\gamma) \geq \girth(\Gamma) \cdot \frac{2}{L}.
\label{eq:gamma_case1}     
\end{equation}

\begin{lem} Any arc in $T$ that connects two distinct slits with each other has length greater or equal to $2/L$. 
\label{lem:slit_connect}
\end{lem}

\begin{proof} In the base rectangle $R$ (see Figure \ref{fig:rect_const}) the distance from the slit to the boundary equals $1/L$. \end{proof}

\item[Case 2.)] $\xi(\gamma)$ is homotopic to a point path:
\begin{itemize}
\item[Case 2.1)] $\xi(\gamma)$ itself is not a point path: This can, for example, happen, if a loop goes through a few tori $T_i$ and then returns to its starting torus. In this case $\gamma$ contains two subarcs $c$ and $c'$ with disjoint interiors, such that $c$ lies in some $T_i$ having its endpoints on a slit $\sigma$ of $T_i$ and $c'$ lies in some other $T_j$ having its endpoints on a slit $\sigma'$ of $T_j$. We claim that 
\begin{equation}
\ell(c) \geq L  \text{ \ and \ } \ell(c') \geq L.
\label{eq:length_c'}
\end{equation}

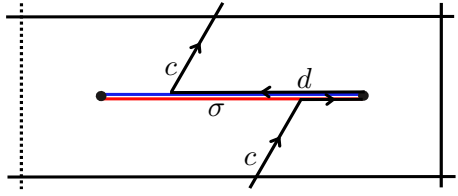
\begin{figure}[!ht]
\SetLabels
\L(.44*.61) $\,c$\\
\L(.48*.39) $\sigma$\\
\L(.55*.54) $\,d$\\
\L(.51*.13) $c$\\
\endSetLabels
\AffixLabels{%
\centerline{%
\includegraphics[height=2.5cm,width=6cm]{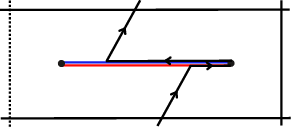}}}
\caption{The auxiliary curve $d$ on the slit $\sigma$ connects two arcs of $c$.}
\label{fig:geod_c}
\end{figure} 

\begin{proof}[Proof of Inequality \eqref{eq:length_c'}] Consider w.l.o.g. $c$ on $T_i$ with endpoints on the slit $\sigma$. If $c$ contains a subarc of positive length on one of the sides of $\sigma$ then, being a geodesic $c$ covers the entire side and has length greater or equal to $L$. Now assume that this is not the case. Then again, because $c$ is a geodesic, it is not homotopic with fixed endpoints to an arc that lies entirely on $\sigma$. Connect the endpoint of $c$ with the initial point by a geodesic arc $d$ of length $\ell(d) \leq L$ that lies on $\sigma$ (Figure \ref{fig:geod_c} shows a case). Then the curve $c$ followed by $d$ is essentially simple and not homotopic to a point curve.

\textit{Essentially simple} means that possible double points which may occur at the cone points can be removed by small homotopies. The inequality $\ell(c) \geq L$ now follows from the next lemma.  

\begin{lem} Let $\eta$ be an essentially simple closed curve on the slit torus $T$ that is not homotopic to a point. Then $\ell(\gamma) \geq 2L$. 
\label{lem:length_essential}
\end{lem}

\begin{figure}[!ht]
\SetLabels
\L(.20*.91) $\overline{T}$\\
\L(.13*.56) $D$\\
\L(.13*.41) $\eta_1$\\
\L(.29*.56) $a$\\
\L(.21*.22) $b$\\
\L(.49*.91) $\,\overline{T}$\\
\L(.45*.41) $\eta_2$\\
\L(.58*.56) $a$\\
\L(.50*.22) $\,b$\\
\L(.79*.91) $\overline{T}$\\
\L(.76*.41) $\,\,\eta_3$\\
\L(.88*.56) $a$\\
\L(.80*.22) $b$\\
\endSetLabels
\AffixLabels{%
\centerline{%
\includegraphics[height=5.5cm,width=14cm]{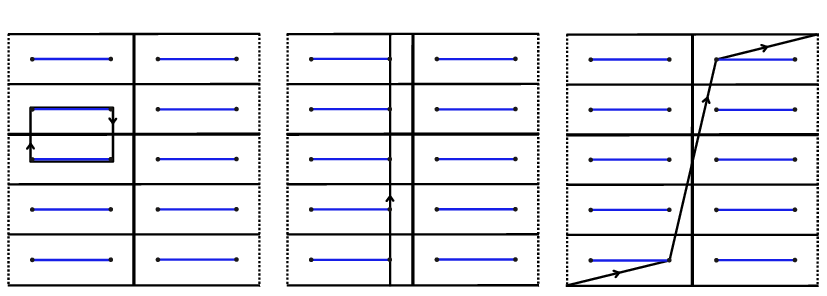}}}
\caption{Non-trivial simple closed curves of $T$ in $\overline{T}$.}
\label{fig:geod_torus}
\end{figure}
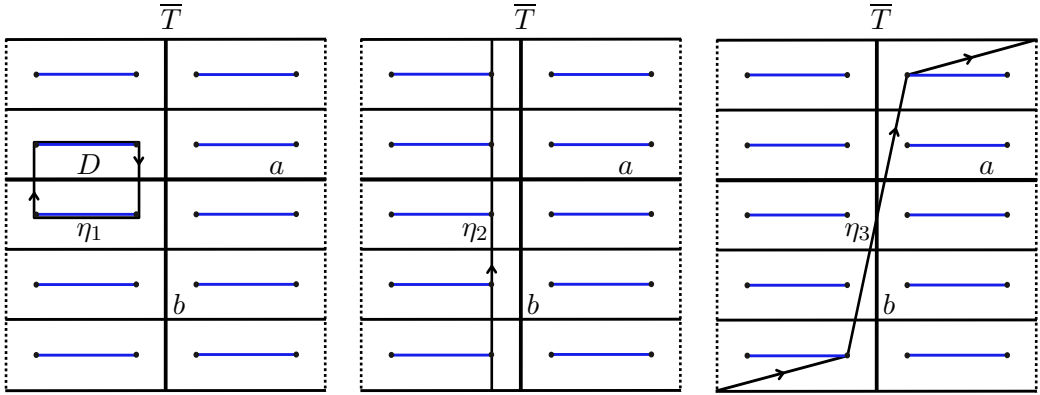 
\begin{proof} Let $\overline{T}$ be the torus obtained from $T$ by closing the opened slits again. We denote by $a$ and  $b$ generators of the fundamental group of $\overline{T}$ represented by closed geodesics of length $\ell(a) = 2L+ 4/L$ and $\ell(b) = 2L$ (see Figure \ref{fig:geod_torus}). In $\overline{T}$ the curve $\eta$ is either homotopic to a point, or else homotopic to some $a^m\cdot b^n$ with $|m| + |n| >0$. If it is homotopic to a point, then, by the Jordan curve theorem it bounds a simply connected domain $D$ as in the example of $\eta_1$ in Figure \ref{fig:geod_torus}. Since $\eta$ is not contractible in $T$, $D$ contains at least one of the sewed slits and so $\ell(\eta) \geq 2L$. If $\eta$ is homotopic to $a^m \cdot b^m$, then its vertical projection with respect to the Euclidean geometry of $T$ has length greater or equal to $|m|\cdot (2L + 4/L)$ and its horizontal projection to $b$ has length greater or equal to $|n| \cdot 2L$. Since $|m| + |n| \neq 0$ this implies that $\ell(\eta) \geq 2L$. This completes the proof of Lemma \ref{lem:length_essential}.   
\end{proof}
Together with the lemma, inequality \eqref{eq:length_c'} follows. \end{proof}

\item[Case 2.2)] $\xi(\gamma)$ is a point path. In this case $\gamma$ runs entirely in some $T_i$ and Lemma \ref{lem:length_essential} implies that $\ell(\gamma) \geq 2L$. 
\end{itemize}

\end{itemize}   
In total we have shown that $\sy(S) = 2L$ and this length is realized by the systoles of the tori $T_i$ in $S$. This concludes the proof of Lemma \ref{lem:sys2L}.
\end{proof}

\subsection{Length estimates} \label{subsec:length}

In order to get $S$ with a large systole we need a regular graph $\Gamma$ with a large \textit{girth}. At the same time $S$ should have "small area". So $\Gamma$ should have a large girth and a relatively small number of vertices. Several constructions are known (\cite{lps},\cite{ls}). The best result for our purpose is from \cite{es}. There, for any $k \geq 2$ and $l \geq 3$ a $k$-regular graph $\Gamma_{k,l}$ of girth $l$ with at most $n(k,l)$ vertices is constructed where 

\begin{equation*}
n(k,l) \leq 4 \sum_{t=1}^{l-2} (k-1)^t.
\end{equation*}
For $k\geq 4$ we have the simplification 
\begin{equation*}
n(k,l) \leq (k-1)^l, \quad (k\geq 4).
\end{equation*}
We recall, that by \eqref{eq:girth_systorus} our goal is to find $\Gamma$, such that
\[
        \girth(\Gamma) \cdot \frac{2}{L} \geq 2L \text{ \ or \ }   \girth(\Gamma)  \geq L^2.
\]
Let now $L \in \N$, $L \geq 2$, take the above slit torus $T$ with $k = 2L^2$ slits (see \eqref{eq:k_Lsquared}) and set $l = L^2$. The graph $\Gamma_L := \Gamma_{2L^2,L^2}$ then has girth $L^2$ and $n_L:=n(2L^2,L^2) \leq (2L^2 -1)^{L^2}$ vertices. By Lemma \ref{lem:sys2L} the surface $S_L$ resulting from pasting $n_L$ copies of $T$ together with pasting pattern $\Gamma_L$ has systole $\sy(S_L) = 2L$.

The genus of $S_L$ can be found using the Euler characteristic that relates the genus  to the cone angles. If a translation surface of genus $g$ has cone points $(p_i)_{i=1 ,\ldots, N}$ with cone angle $2\pi \cdot (k_i+1)$ at $p_i$, where $k_i \in \N$, then
\begin{equation}
      \sum_{i=1}^N k_i = 2g-2.
\label{eq:Euler}
\end{equation}

By this equation and the fact that all cone points of $S_L$ are of angle $4\pi$ we have
\begin{eqnarray}
\nonumber
\area(S_L) &=& n_L(4L^2 +8) \\
g_L := \genus(S_L) &=& n_L L^2 +1. 
\label{eq:area_genus}
\end{eqnarray}
Hence, for any $L \in \N$, $L\geq 2$ we have constructed a translation surface $S_L$ with systole $\sy(S_L) = 2L$ and genus
\begin{equation}
g_L \leq L^2 (2L^2-1)^{L^2} + 1.
\end{equation}
Furthermore, for large $L$ the area $\area({S_{L}})$ is close to $4g$. Theorem \ref{thm:sr_low} now follows from the next lemma. 
\hfill $\square$
\begin{lem}\label{lem:inverse}
If $4 < y \leq x (2x)^x$ and $x > 6$, then $\frac{\log(y)}{\log \log(y)} < x$.
\end{lem}
\begin{proof} It suffices to prove the lemma for the case of equality $y = x(2x)^x$. Since then $\log(y) = \log(x) + x \log(2x))$ the statement for $x$ is equivalent to
\begin{equation*}
\log(x) + x \log(2x) \leq x \log( \log(x) + x \log(2x)).
\end{equation*}
For $6 < x < 12$ this inequality may be checked by plotting the two functions of $x$. For $x \geq 12$ we argue that then
\begin{equation*}
\log(x) + x \log(2x) < x \log(3x) \quad \text{and} \quad 3x < \log(x) + x \log(2x).
\end{equation*}

\vspace{-18pt}

\end{proof}
We add that in Lemma \ref{lem:inverse} we also have the upper bound $x < \log(y)$, for $x>6$.

\section{Short loops on hyperelliptic translation surfaces} \label{sec:loops}

In this section we apply the methods used in \cite{bmm}  in order to show that, in the same way as for hyperbolic surfaces, large systoles on translation surfaces do not occur if they are hyperelliptic. As only topological and area arguments shall be used for this we consider, more generally, any compact orientable surface endowed with a Riemannian metric of non-positive curvature that has finitely many cone-like singularities all of whose cone angles are greater than $2\pi$. We call them \emph{cn(0) surfaces} (short for ``two dimensional cone manifold of curvature $\leq0$''). Any translation surface is a cn(0) surface. The relevant property of a cn(0)-surface is that any embedded open disk $B_r(q)$ of radius $r$ centered at some point $q$ has area
\begin{equation}\label{eq:areaB}
\area ( B_r(q)) \geq \pi \, r^2.
\end{equation}
A cn(0) surface $S$ of genus $g$ is called \emph{hyperelliptic} if there exists an isometry $\phi : S \to S$ of order two, called the \emph{hyperelliptic involution}, that has $2g+2$ fixed points. The fixed points are called the \emph{Weierstrass points}; some of the Weierstrass points may be at the same time cone points.

We first restrict ourselves to the systole for which the arguments are fairly straightforward. The normalisation of the area to $4\pi g$ in the following theorem is made to simplify the calculations. The theorem is stated for the homological systole $\sy_h(S)$.

\begin{thm}[Homology systoles of hyperelliptic surfaces] Let $S$ be a hyperelliptic cn(0) surface of genus $g \geq 2$ and area $4\pi g$. Then %
\begin{equation*}
\sy_h(S)  \leq  4\sqrt{2} \text{ \ and \ } \frac{\sy_h(S)^2}{\area(S)} \leq \frac{8}{\pi g}.
\end{equation*}
\label{thm:hyperell_systole}
\end{thm}
\begin{proof}
We formulate the proof in such a way that it also serves as Step 1 in the proof of the next theorem. Let $p_1^*, \ldots, p_{2g+2}^*$ be the Weierstrass points fixed by the hyperelliptic involution $\phi : S \to S$. The quotient $\Sigma = S/\phi$ under the action of $\phi$ is a topological sphere, we denote by $p_1, \ldots, p_{2g+2}$ the images of $p_1^*, \ldots, p_{2g+2}^*$ under the natural projection $\Pi : S \to \Sigma$ (see Figure \ref{fig:hyp}). These too shall be called \emph{Weierstrass points}. 
Since $\phi$ acts as an isometric involution on each $B_r(p_i^*)$ the metric disks $B_r(p_i)$ on $\Sigma$ have half the areas of the $B_r(p_i^*)$ on $S$. Thus, by \eqref{eq:areaB}, as long as for given $i$ the radius $r$ is such that $B_r(p_i)$ is embedded,
\begin{equation}\label{eq:areaBhalf}
\area ( B_r(p_i)) = \frac{1}{2} \area ( B_r(p_i^*)) \geq \frac{\pi}{2}\cdot r^2, \quad i=1, \ldots, 2g+2.
\end{equation}

\begin{figure}[!ht]
\SetLabels
\L(.43*.25) $p_1$\\
\L(.42*.77) $p_1^*$\\
\L(.31*.25) $p_2$\\
\L(.35*.77) $p_2^*$\\
\L(.55*.25) $p_{10}$\\
\L(.66*.25) $p_9$\\
\L(.78*.25) $p_8$\\
\L(.88*.25) $p_7$\\
\L(.95*.25) $p_6$\\
\L(.02*.25) $p_5$\\
\L(.09*.25) $p_4$\\
\L(.20*.25) $p_3$\\
\L(.52*.75) $\beta$\\
\L(.36*.90) $\alpha'$\\
\L(.36*.65) $\alpha''$\\
\L(.42*.12) $\delta_1$\\
\L(.45*.01) $q$\\
\L(.46*.74) $\delta^*_1$\\
\L(.49*.50) $S$\\
\L(.48*.35) $\Sigma$\\
\endSetLabels
\AffixLabels{%
\centerline{%
\includegraphics[height=6cm,width=15cm]{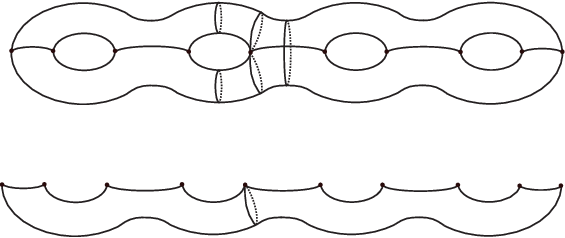}}}
\caption{A hyperelliptic surface $S$ of genus four with ten Weierstrass points.}
\label{fig:hyp}
\end{figure}
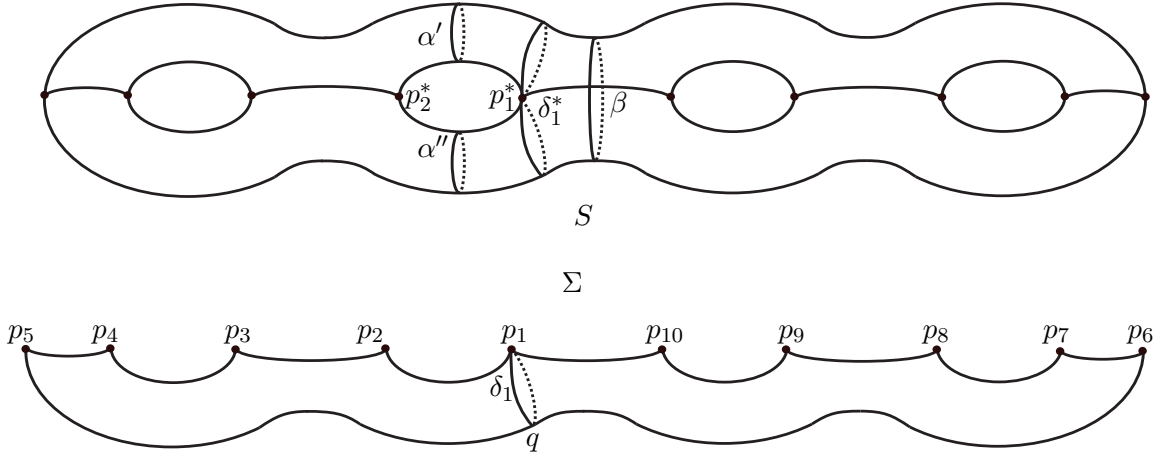

The strategy is to find a short geodesic loop at some $p_i$ on $\Sigma$ and lift it to a short homologically non-trivial closed curve on the covering surface $S$. For this we look at the union of disks \mbox{$B_r(p_1) \cup \ldots \cup B_r(p_{2g+2}) $} on $\Sigma$ and let $r$ grow continuously beginning with $r=0$ until one of the following two cases occurs for the first time:

\medskip
Case A)  \emph{the closure of a single disk comes to self-intersect.}

Case B)  \emph{the closures of two different disks come to intersect.}

\medskip

Let $r_1$ be the smallest value of $r$ for which this occurs. Then the open disks $B_{r_1}(p_1), \ldots, B_{r_1}(p_{2g+2})$ are embedded and pairwise disjoint but not so all of the closed ones. Since the sum of the areas of the disjoint disks is less than $\area(\Sigma) = \frac{1}{2} \area(S) = 2\pi g$ we have by \eqref{eq:areaBhalf},
$(2g+2) \frac{\pi}{2}r_1^2 \leq 2\pi g$ which yields
\begin{equation}
r_1 \leq \sqrt{\frac{2g}{g+1}} < \sqrt{2}.
\end{equation}
Now suppose it is Case A) that is occurring for $r_1$, say for the disk at $p_1$ (see Figure \ref{fig:hyp}). Then there are two geodesic arcs of length $r_1$ emanating from $p$ meeting at their endpoints in some point $q$. This point may be an ordinary point or a cone point. Since the open disk $B_{r_1}(p_1)$ is embedded the arcs meet at an angle $\pi$ provided $q$ is an ordinary point respectively, in such a way that both angles at $q$ are $\geq \pi$ if $q$ is a cone point. In either case the two arcs form a geodesic loop $\delta_1$ (see Figure \ref{fig:hyp}). It was shown in \cite{mu1} that $\delta_1$ lifts  to a figure eight closed geodesic $\delta^*_1$ of length $2\ell(\delta_1) = 4r_1$ with self intersection at the lift $p^*_1$ of $p_1$. Splitting $\delta^*_1$ at $p^*_1$ into two loops we get two non-separating simple closed geodesics $\alpha'$ and $\alpha''$ in the free homotopy classes of these loops, where 
\begin{equation*} 
    \ell(\alpha') = \ell(\alpha'') \leq \ell(\delta_1) = 2r_1. 
\end{equation*}
Note that the homotopy classes of $\alpha'$ and $\alpha''$ may be identical. We also mention that the figure eight geodesic can be split in another way so as to get a simple closed homologically trivial geodesic $\beta$.

We turn to Case B). Here the closures of two distinct disks, say $B_{r_1}(p_1)$ and $B_{r_1}(p_2)$ meet at some point $q$. By the same argument as before the two geodesic arcs of length $r_1$ from $p_1$ to $q$ and from $p_2$ to $q$ together form a geodesic arc $\delta_1$ from $p_1$ to $p_2$. In this case $\delta_1$ lifts to a homotopically non-trivial $\phi$-invariant simple closed geodesic $\alpha$ of length
\begin{equation*}
 \ell(\alpha) = 2 \ell(\delta_1) = 4r_1.
\end{equation*}
This completes the proof.
\end{proof}

\begin{thm}[Homology bases of hyperelliptic surfaces] Let $S$ be a hyperelliptic cn(0) surface of genus $g \geq 2$ and area $4\pi g$. Then there exist $\lceil  \frac{2g+2}{3}  \rceil$ geodesic loops  $(\alpha_k)_{k=1,\ldots,\left \lceil \frac{2g+2}{3}\right \rceil}$ that can be extended to a homology basis of $H_1(S,\Z)$ such that
\begin{equation*}
\ell(\alpha_{k}) \leq 4 \cdot \sqrt{\frac{ 12g}{ 2g+5-3k}} \text{ \ and \ } \frac{\ell(\alpha_{k})^2}{\area(S)} \leq  \frac{ 48}{ \pi(2g+5-3k)} \text{ for all } k=1, \ldots, \left\lceil \frac{2g+2}{3}\right \rceil.
\end{equation*}
\label{thm:hyper_section}
\end{thm}
This theorem implies Theorem \ref{thm:hyper_intro} from the introduction.

\begin{proof}
The proof has two parts. In the first part we construct a collection of geodesic arcs on $\Sigma$ whose endpoints are Weierstrass points. This is carried out by estimating  areas of growing disks extending the procedure in the proof of Theorem \ref{thm:hyperell_systole}. The lifts of these arcs in $S$ form a family of simple closed curves, however, this family may not be extendable to a homology basis. In the second part of the proof we shall therefore restrict the family to a subset so as to make the extension possible. The second part uses only topological arguments.

\medskip

\emph{Part 1.}\ We let the common radius $r$ of the disks $B_r(p_i)$ grow but stepwise prevent certain disks from further growing. Step 1 has been carried out in the proof of Theorem \ref{thm:hyperell_systole} at the end of which either Case A) or Case B) has occurred for some radius $r = r_1$. In Case A) the open disk $B_{r_1}(p_1)$ is still embedded but its boundary has a tangential self intersection and we have detected a geodesic loop $\delta_1$ of length $2r_1$ with base point $p_1$. In this case we \emph{freeze} disk $B_{r_1}(p_1)$, meaning we now prevented it from growing further.

In Case B) two disjoint embedded open disks $B_{r_1}(p_1)$, $B_{r_1}(p_2)$ are detected whose boundaries intersect each other tangentially and a simple geodesic arc $\delta_1$ of length $2r_1$ from $p_1$ to $p_2$ is detected. In this case we freeze $B_{r_1}(p_1)$ and $B_{r_1}(p_2)$ i.e we prevent both of them from growing further.

We now continue the procedure letting the common radius of the non frozen disks grow beyond $r_1$ until again Case A) or Case B) occurs, at which moment we get a new geodesic arc, freeze the involved disks, and so on. If several disks happen to come to touch at the same time, we can choose an arbitrary order to turn such a case into a succession of individual steps with only one new arc in each. We detail the general step for $m \geq 2$.

\medskip

\textbf{Step \!$\boldsymbol{m}$:}  Let $j_{m-1}$ denote the number of all disks that are frozen at the end of Step\ $m{-}1$. If $j_{m-1} = 2g+2$ the procedure has come to halt at the end of the preceding step and we note $M:=m-1$. Otherwise, we expand the  remaining disks $(B_r(p_i))_{i= j_{m-1} +1 ,\ldots, 2g+2}$, letting the common radius $r$ grow further until again for some $r=r_m$ one of the following occurs:

\medskip
Case A)  \emph{the closure of a single disk comes to self-intersect.}

Case B)  \emph{the closures of two different disks come to intersect.}

\medskip
In Case A) we we set $j_m = j_{m-1}+1$, assume, by adapting the notation, that the center of the said disk is $p_{j_m}$, and get a simple geodesic loop $\delta_m$ of length $\ell(\delta_m) = 2r_m$ based at $p_{j_m}$, where $r_m$ is the radius currently reached. Disk $B_{r_m}(p_{j_m})$ gets frozen.

In Case B) for $r =r_m$ two subcases may occur: B1)  an expanding disk w.l.o.g. centered at $p_{j_{m-1}+1}$ touches an already frozen disk in which case we set $j_m = j_{m-1}+1$ and freeze $B_{r_m}(p_{j_m})$ or, B2) two expanding disks w.l.o.g. centered at $p_{j_{m-1}+1}$ and $p_{j_{m-1}+2}$ come to touch for $r=r_m$ in which case we set $j_m = j_{m-1}+2$ and freeze both, $B_{r_m}(p_{j_m -1})$ and $B_{r_m}(p_{j_m})$. In either case we get a simple geodesic arc $\delta_m$ from $p_{j_m}$  to $p_{j_m - k}$ for some $k$.

The procedure stops at the end of Step $M$ at which moment it has produced simple geodesic arcs and loops $\delta_1, \ldots, \delta_M$ that pairwise intersect each other at most at their endpoints, with $g+1 \leq M \leq 2g+2$.

To estimate the lengths of the $\delta_m$ we note that by \eqref{eq:areaBhalf} we have for the $2g+2 - j_{m-1}$ growing disks in Step\ $m$ at the moment when $r = r_m$,
\begin{equation*}
(2g+2 - j_{m-1})\frac{\pi r_m^2}{2} \leq \sum_{i=j_{m-1}+1}^{2g+2}\hspace{-1ex}\area(B_{r_m}(p_i)) \leq \area(\Sigma) = 2\pi g.
\end{equation*}
This yields
\begin{equation}             
 \frac{ \ell(\delta_m) }{2}  = r_m \leq \sqrt{\frac{ 4g }{ 2g+2-j_{m-1}} }, \quad m=1,\ldots, M
\label{eq:deltam}
\end{equation}

(with $j_0=0$). We finish Part 1 by analysing the combinatorial properties of the graph formed by $\delta_1, \ldots, \delta_M$. More precisely, for any $m \in \{1, \ldots, M\}$, we let $\cH_m$ be the graph embedded in the sphere $\Sigma$ whose vertex set is the set of Weierstrass points $p_1, \ldots, p_{2g+2}$ and the edges are $\delta_1, \ldots, \delta_m$. In the following \textit{loop} means closed edge. 

\begin{lem} Each $\cH = \cH_m$ has the following properties.
\begin{enumerate}
\item[(i)] $\cH$ is embedded in $\Sigma$.
\item[(ii)] The vertices of $\cH$ are the Weierstrass points $p_1, \ldots, p_{2g+2}$.
\item[(iii)] For any loop $\lambda$ of $\cH$ each of the two open components of $\Sigma \setminus \lambda$ contains at least one vertex of $\cH$.
\item[(iv)] Each connected component of $\cH$ is one of the following
\begin{itemize}
\item[--] an isolated vertex,
\item[--] a non-trivial tree,
\item[--] an isolated loop,
\item[--] a non-trivial tree with exactly one loop attached to it.
\end{itemize}
\end{enumerate}
\label{lem:graph_nooroneloop}
\end{lem}
Here \emph{non-trivial} means that the tree has at least one edge. In what follows, \emph{any} graph $\cH$ with properties (i) - (iv) will be called a \emph{graph of type W} (the $W$ alludes to the Weierstrass points).

\begin{proof} Properties (i) and (ii) hold by the definition of the $\cH_m$. Property (iii) follows from the Gauss-Bonnet theorem since the curvature of the metric we are dealing with is $\leq 0$ away from the Weierstrass points. For the properties of the components in (iv) we proceed by induction on the stepwise search process of the $\delta_m$.

At the end of Step 1 we have graph $\cH_1$ whose components are vertices plus one edge that is either closed or connects distinct vertices.

Now assume that the lemma holds for the graph $\cH_{m-1}$ which is given at the beginning of Step\ $m$. Let $\delta_m$ be the new geodesic arc obtained at this step. By construction the initial point $p_{j_m}$ of $\delta_m$ is a Weierstrass point to which no edge of $\cH_{m-1}$ is attached. For the other endpoint of $\delta_m$  there are two possibilities: 1.) it coincides with $p_{j_m}$ in which case $\delta_m$ together with $p_{j_m}$  is a component of $\cH_{m}$, or 2.) $\delta_m$ connects $p_{j_m}$ with a component $\cT$ of $\cH_{m-1}$ in which case $\cT' := \cT + \delta_m + p_{j_m}$ is a tree in case $\cT$ is an isolated vertex or a tree, respectively a tree  with exactly one closed attached edge in case $\cT$ is a loop or a tree with an attached loop. Hence, $\cH_m$ too satisfies (i) -- (iv).
\end{proof} 
We add that, by construction, the final graph $\cH_M$
has no isolated vertices.

\bigskip

\emph{Part 2.}\ The lifts of $\delta_1, \ldots, \delta_M$ to $S$ under the hyperelliptic covering $\Pi : S \to \Sigma$ are generally not extendable to a basis of $H_1(S,\Z)$. We therefore remove certain parts to remedy this. The arguments make no use of metric geometry. They are carried out in \cite{bmm} and we just mention the main steps for clarity.

Two types of removals come into play. One is to remove edges from $\cH_M$, the other takes place on $S$ and is as follows: Let $\cH$ be a subgraph of $\cH_M$ obtained by removing some edges but leaving the vertex set unchanged. Then $\cH$ is again of type $W$. The inverse image $\cH^*$ of $\cH$ under the covering map $\Pi : S \to \Sigma$ consists of simple closed curves (the lifts of the edges of $\cH$ that are not loops) and figure eight shaped curves (the lifts of the loops of $\cH$). We now modify $\cH^*$ by removing from each figure eight shape one of the two simple loops it consists of (it does not matter which). The thus modified lift shall be denoted by $\cH^\#$. It consists of simple closed curves that pairwise intersect each other in at most one point.

The goal is now to remove certain edges from $\cH_M$, as few as possible, so that $\cH^\#$ extends to a basis. For concise expression we say that closed curves $a_1, \ldots, a_n$ form a \emph{partial basis} if their homology classes $[a_1], \ldots, [a_n]$ can be extended to a homology basis. The following sufficient criterion is used \cite[Lemma 3.3]{bmm}:
\begin{lem}
Let $a_1, \ldots, a_{n}$ be distinct simple closed curves on a compact orientable surface $F$ of genus $g \geq 1$. If $F \smallsetminus (a_1 \cup \ldots \cup a_{n})$ is connected then the curves form a partial basis of $\smash H_1(F,\Z)$.
\label{lem:rank_n}
\end{lem}
It suffices therefore to reduce $\cH_M$ to $\cH$ in such a way that $S \setminus \cH^\#$ is connected. For this we have another criterion \cite[Lemma 3.2]{bmm}:
\begin{lem} $\cH^\#$ in $S$ is non-separating if and only if any open connected component of $\Sigma \smallsetminus \cH$ admits a Jordan curve $\mathcal{J}$ that separates $p_1, \ldots, p_{2g+2}$ into two odd subsets i.e.\ the number of Weierstrass points on either side of $\mathcal{J}$ is odd.
\label{lem:lift_nonsep}
\end{lem}
The property expressed in this lemma shall be called the \emph{Weierstrass point separation property}.
It thus remains to delete edges from $\cH_M$ so that the resulting graph $\cH$ has this property. This is carried out by a pruning algorithm in \cite{bmm} which we do not reproduce here but merely state its outcome in form of a lemma. At one point in the description of the algorithm in \cite{bmm} a curvature argument is made to show that on either side of any loop of $\cH_M$ on $\Sigma$ there lies at least one Weierstrass point. Here we have shifted this property to item (iii) of Lemma \ref{lem:graph_nooroneloop} above and included it in the definition of a graph of type $W$ so as to free the next lemma from metric hypotheses.

The term ``edges'' in the lemma englobes open and closed ones.
\begin{lem} Let $\cG$ be a graph of type $W$ embedded in the sphere $\Sigma$ with 2g+2 vertices and assume that none of the vertices is isolated. Then $\cG$ contains a subgraph $\cH$ of type $W$ with $\left\lceil \frac{2g+2}{3}\right \rceil$ edges that has the Weierstrass point separation property.
\label{lem:frompruning}
\end{lem}
We now apply the lemma to $\cG = \cH_M$, where we recall that $\cH_M$ has no isolated vertices. The edges $\delta_m$ of $\cH_M$, for $m=1, \ldots, M$,  have the length bounds \eqref{eq:deltam}. Lemma \ref{lem:frompruning} yields a selection
\begin{equation*}
\delta_{m_1},  \ldots, \delta_{m_K},  \quad K:= \left\lceil \frac{2g+2}{3}\right \rceil,
\end{equation*}
that form a subgraphs $\cH$ of type $W$ that has the Weierstrass points separation property. We chose the order of the indices $m_1, \ldots, m_K$ in such a way that
\begin{equation*}
j_{m_1-1} < j_{m_2-1} < \ldots < j_{m_K-1} < 2g+2.
\end{equation*}
We then have 
\begin{equation*}
j_{m_k-1} \leq 2g-K+1+k, \quad k=1,\ldots,K.
\end{equation*}
By \eqref{eq:deltam} and the definition of $K$ we get
\begin{equation*}
\ell(\delta_{m_k}) \leq 2 \sqrt{\frac{4g}{K+1-k}} \leq 2 \sqrt{\frac{12g}{2g+5-3k}}, \quad k = 1, \ldots, \left\lceil \frac{2g+2}{3}\right \rceil.
\end{equation*}
The lifts of $\delta_{m_1}, \ldots, \delta_{m_K}$ in $\cH^\#$ have at most twice the lengths and by Lemma \ref{lem:rank_n} and \ref{lem:lift_nonsep} they form a partial basis. This completes the proof.
\end{proof}

\vspace{1cm}

\noindent Peter Buser \\
\noindent Department of Mathematics, Ecole Polytechnique F\'ed\'erale de Lausanne\\
\noindent Station 8, 1015 Lausanne, Switzerland\\
\noindent e-mail: \textit{peter.buser@epfl.ch}\\
\\
\\
\noindent Eran Makover\\
\noindent Department of Mathematics, Central Connecticut State University\\
\noindent 1615 Stanley Street, New Britain, CT 06050, USA\\
\noindent e-mail: \textit{makovere@ccsu.edu}\\
\\
\\
\noindent Bjoern Muetzel \\
\noindent Department of Mathematics, Eckerd College \\
\noindent 4200 54th avenue South, St. Petersburg, FL 33711, USA\\
\noindent e-mail: \textit{bjorn.mutzel@gmail.com}\\


\begin{thebibliography} {[W/M/W]}
\normalsize
\bibitem[Ak]{ak}
Akrout, H.: \textit{Singularit\'es topologiques des systoles g\'en\'eralis\'ees}, Topology {\bf42}(2) (2003), 291--308.

\bibitem[AM1]{am1}
Akrout H. and Muetzel B. : \textit{Construction of Riemann surfaces with large systoles}, Journal of Geometry {\bf 107} (2016), 187--205.

\bibitem[AM2]{am2}
Akrout, H. and Muetzel, B.: \textit{Construction of surfaces with large systolic ratio}, \textit{arXiv:1311.1449} (2018).

\bibitem[Ba]{ba}
Bavard, C.: \textit{Systole et invariant d'Hermite} J. Reine. Angew. Math. {\bf482} (1997), 93--120. 

\bibitem[Bu]{bu}
Buser, P.: \textit{Riemannsche Fl\"achen mit grosser Kragenweite} (German), Comment. Math. Helv. {\bf53}(1) (1978), 395--407.

\bibitem[BMM]{bmm} Buser, P., Makover, E. and Muetzel B.: \textit{Short homology basis for hyperelliptic hyperbolic surfaces}, Isr. J. Math. (2023), https://doi.org/10.1007/s11856-023-2600-y.

\bibitem[CHMW]{chms} 
Columbus T., Herrlich F., Muetzel B. and Weitze-Schmith\"usen G.: \textit{Systolic geometry of translation surfaces}, Exp. Math. (2022), doi:10.1080/10586458.2022.2108946.

\bibitem[ES]{es}
Erd{\"o}s P. and Sachs H.: \textit{Regul\"are Graphen gegebener Taillenweite mit minimaler Knotenzahl}, Wiss. Z. Martin-Luther-Univ. Halle-Wittenberg
Math.-Naturwiss. Reihe {\bf12} (1963), 251--257.


\bibitem[Ge] {ge}
  Gendulphe, M.: \textit{D\'ecoupages et in\'egalit\'es systoliques pour les surfaces hyperboliques \`{a} bord}, Geometriae dedicata {\bf142} (2009), 23--35.

\bibitem[JP]{jp} 
Judge, Ch. and Parlier H.: \textit{The maximum number of systoles for genus two Riemann surfaces with abelian differentials}, Comment. Math. Helv. {\bf94}(2)  (2019), 399--437. 

\bibitem[KS]{ks} Katz, M. and Sabourau, S.: \textit{Entropy of systolically extremal surfaces and
asymptotic bounds}, Ergo. Th. Dynam. Sys. {\bf25}(4) (2005), 1209--1220.

\bibitem[KZ]{kz}
Kontsevich, M. and Zorich A.: \textit{Connected components of the moduli spaces of abelian differentials with prescribed singularities}, Invent. Math. {\bf153} (2003), 631--678.

\bibitem[LPS]{lps}
Lubotzky, A., Phillips, R. and Sarnak, P.: \textit{Ramanujan graphs}, Combinatorica {\bf8} (1988), 261--277.

\bibitem[LS]{ls}
Linial, N. and Simkin M.: \textit{A randomized construction of high girth regular graphs}, Random Structures \& Algorithms {\bf58}(2) (2021), 345--369.

\bibitem[Ma]{ma}
Massart, D.: \textit{A short introduction to translation surfaces, Veech surfaces, and Teich-
mueller dynamics}, In: Surveys in Geometry I, Springer (2022), 343--388. 

\bibitem[Mu]{mu1}
Muetzel, B.: \textit{On the second successive minimum of the Jacobian of a Riemann surface}, Geom. Dedicata {\bf161}(1) (2012), 85--107.

\bibitem[Pa]{pa1}
Parlier, H.: \textit{The homology systole of hyperbolic Riemann surfaces}, Geom. Dedicata {\bf157}(1) (2012), 331--338.

\bibitem[PW]{pw} 
Petri, B. and Walker A.: \textit{Graphs of large girth and surfaces of large systole}, Math. Res. Lett. {\bf25}(6) (2018), 1937--1956. 

\bibitem[Ra]{ra}
Rafi, K.: \textit{Thick-thin decomposition for quadratic differentials},
Math. Res. Lett. {\bf14}(2) (2007), 333--341.

\bibitem[Sc1]{sc1} 
Schmutz, P.: \textit{Congruence subgroups and maximal Riemann surfaces}, J. Geom. Anal. {\bf4}(2) (1994), 207--218.

\bibitem[Sc2]{sc2}
Schmutz Schaller, P.: \textit{Riemann surfaces with shortest geodesic of maximal length}, Geom. Funct. Anal. {\bf3}(6) (1993), 564--631.

\bibitem[Wr]{wr}
  Wright, A.: \textit{Translation surfaces and their orbit closures: an introduction for a broad audience}, EMS Surv. Math. Sci. {\bf 2}(1) (2015), 63--108.


\bibitem[Zo]{zo}
Zorich, A.: \textit{Flat surfaces}. Frontiers in number theory, physics, and geometry I,  Springer, Berlin (2006), 437--583.


\bibliographystyle{plain}
\end{thebibliography}
\end{document}